\documentclass[12pt, reqno]{amsart}

\author[M.~Balcerzak]{Marek Balcerzak}
\address{Institute of Mathematics, Lodz University of Technology, al. Politechniki 8, 93-590 Lodz, Poland}
\email{marek.balcerzak@p.lodz.pl}
\author[S.~G\l \k{a}b]{Szymon G\l \k{a}b}
\address{Institute of Mathematics, Lodz University of Technology, al. Politechniki 8, 93-590 Lodz, Poland}
\email{szymon.glab@p.lodz.pl}
\author[P.~Leonetti]{Paolo Leonetti}
\address{Department of Economics, Universit\`a degli Studi dell'Insubria, via Monte Generoso 71, 21100 Varese, Italy}
\email{leonetti.paolo@gmail.com}



\title{Topological complexity of ideal limit points}

\usepackage[T1]{fontenc}
\usepackage{amsmath}
\usepackage{amssymb}
\usepackage{amsthm}
\usepackage[left=3.5cm, right=3.5cm]{geometry}
\usepackage{hyperref}
\usepackage{fancyhdr}
\usepackage{enumitem}
\usepackage{comment}
\usepackage{nicefrac}
\usepackage{comment}
\usepackage{bm}
\usepackage{mathrsfs}
\usepackage{graphicx}
\usepackage[utf8]{inputenc}
\usepackage{cancel}
\usepackage{mathtools}

\usepackage{tikz}
\usetikzlibrary{decorations.pathreplacing,matrix,arrows,positioning,automata,shapes,shadows,calc,fadings,decorations,snakes,through,intersections}

\AtBeginDocument{%
   \def\MR#1{}
}

\newtheorem{thm}{Theorem}[section]
\newtheorem{cor}[thm]{Corollary}
\newtheorem{lem}[thm]{Lemma}
\newtheorem{prop}[thm]{Proposition}

\theoremstyle{definition} 
\newtheorem{defi}[thm]{Definition}
\let\olddefi\defi
\renewcommand{\defi}{\olddefi\normalfont}
\newtheorem{question}{Question}
\let\oldquestion\question
\renewcommand{\question}{\oldquestion\normalfont}
\newtheorem{example}[thm]{Example}
\let\oldexample\example
\renewcommand{\example}{\oldexample\normalfont}
\newtheorem{rmk}[thm]{Remark}
\let\oldrmk\rmk
\renewcommand{\rmk}{\oldrmk\normalfont}
\newtheorem{claim}{\textsc{Claim}}

\hypersetup{
    pdftitle={Topological complexity of the sets of ideal limit points},
    pdfauthor={Balcerzak, Glab, and Leonetti},
    pdfmenubar=false,
    pdffitwindow=true,
    pdfstartview=FitH,
    colorlinks=true,
    linkcolor=blue,
    citecolor=green,
    urlcolor=cyan
}

\providecommand{\MR}[1]{}

\providecommand{\MR}{\relax\ifhmode\unskip\space\fi MR }

\providecommand{\href}[2]{#2}

\begin{document}

\maketitle
\thispagestyle{empty}

\begin{abstract}
Given an ideal $\mathcal{I}$ on the nonnegative integers $\omega$ and a Polish space $X$, let $\mathscr{L}(\mathcal{I})$ be the family of subsets $S\subseteq X$ such that $S$ is the set of $\mathcal{I}$-limit points of some sequence taking values in $X$. First, we show that $\mathscr{L}(\mathcal{I})$ may attain arbitrarily large Borel complexity. Second, we prove that if $\mathcal{I}$ is a $G_{\delta\sigma}$-ideal then all elements of $\mathscr{L}(\mathcal{I})$ are closed. Third, we show that if $\mathcal{I}$ is a simply coanalytic ideal and $X$ is first countable, then every element of $\mathscr{L}(\mathcal{I})$ is simply analytic. Lastly, we studied certain structural properties and the topological complexity of minimal ideals $\mathcal{I}$ for which $\mathscr{L}(\mathcal{I})$ contains a given set. 
\end{abstract}


\section{Introduction}\label{sec:intro}

Let $\mathcal{I}$ be an ideal on the nonnegative integers $\omega$, that is, 
a subset of $\mathcal{P}(\omega)$ closed under taking subsets and finite unions. Unless otherwise stated, it is assumed that $\mathcal{I}$ is admissible, namely, $\omega\notin \mathcal{I}$ and that $\mathcal{I}$ contains the family $\mathrm{Fin}$ of finite subsets of $\omega$. 
Intuitively, the ideal $\mathcal{I}$ represents the family of \textquotedblleft small\textquotedblright\ subsets of $\omega$. An important example is the family of asymptotic density zero sets
$$
\mathcal{Z}:=\left\{S\subseteq \omega: \lim_{n\to \infty} \frac{|S \cap [0,n]|}{n}=0\right\}.
$$
Define $\mathcal{I}^+:=\mathcal{P}(\omega)\setminus \mathcal{I}$. 
Ideals are regarded as subsets of the Cantor space $\{0,1\}^\omega$, hence we can speak about their topological complexity. For instance, $\mathrm{Fin}$ is a $F_\sigma$-ideal, and $\mathcal{Z}$ is a $F_{\sigma\delta}$-ideal which is not $F_\sigma$. 

Pick also a sequence $\bm{x}=(x_n)$ be taking values in a topological space $X$.
Then, we denote by $\Lambda_{\bm{x}}(\mathcal{I})$ the set of $\mathcal{I}$\emph{-limit points} of $\bm{x}$, that is, the set of all $\eta \in X$ for which there exists a subsequence $(x_{n_k})$ such that 
$$
\lim_{k\to \infty} x_{n_k}=\eta
\quad\text{ and }\quad 
\{n_k: k \in \omega\}\in \mathcal{I}^+.
$$
It is well known that, even in the case where ${\bm{x}}$ is a real bounded sequence, it is possible that $\Lambda_{\bm{x}}(\mathcal{Z})$ is the empty set, see \cite[Example 4]{MR1181163}. 
The topological nature of the sets of $\mathcal{I}$-limits points $\Lambda_{\bm{x}}(\mathcal{I})$ and their relationship with the slightly weaker variant of $\mathcal{I}$-cluster points have been studied in \cite{BL}, cf. also \cite{Xi, Xi2}. 

In this work, we continue along this line of research. To this aim, we introduce our main definition:

\begin{defi}\label{defi:setIlimitpoints}
Let $X$ be a topological space and $\mathcal{I}$ be an ideal on $\omega$. We denote by $\mathscr{L}_X(\mathcal{I})$ the family of sets of $\mathcal{I}$-limit points of sequences ${\bm{x}}$ taking values in $X$ together with the emptyset, that is, 
$$
\mathscr{L}_X(\mathcal{I}):=\left\{A\subseteq X: A=\Lambda_{\bm{x}}(\mathcal{I}) \text{ for some sequence }{\bm{x}} \in X^\omega\right\}\cup \{\emptyset\}.
$$
If the topological space $X$ is understood, we write simply $\mathscr{L}(\mathcal{I})$. 
\end{defi}

A remark is in order about the the addition of $\{\emptyset\}$ in the above definition: it has been proved by Meza-Alc\'{a}ntara in \cite[Section 2.7]{Mezathesis} that, if $X=[0,1]$, then there exists a $[0,1]$-valued sequence $\bm{x}$ such that $\Lambda_{\bm{x}}(\mathcal{I})=\emptyset$ if and only if there exists a function $\phi: \omega \to \mathbb{Q}\cap [0,1]$ such that $\phi^{-1}[A] \in \mathcal{I}$ for every set $A\subseteq \mathbb{Q}\cap [0,1]$ with at most finitely many limit points; cf. also \cite[Proposition 6.4]{MR3034318} and, more generally, \cite{FKL24} for analogues in compact uncountable spaces. On the other hand, if $X$ is not compact, it is easy to see that, for every ideal $\mathcal{I}$, there is a sequence with no $\mathcal{I}$-limit points.  
Thus, the addition of $\{\emptyset\}$ in the above definition avoids the repetition of known results in the literature and 
to add further subcases based on the topological structure of the underlying space $X$. 

We summarize in Theorem \ref{thm:oldthm} and Theorem \ref{thm:xi} below the known results from \cite{BL, Xi, Xi2} about the families $\mathscr{L}_X(\mathcal{I})$. 

For, given a topological space $X$ and an ordinal $1\le \alpha <\omega_1$, we use the standard Borel pointclasses notations $\Sigma^0_\alpha(X)$ and $\Pi^0_\alpha(X)$, so that $\Sigma^0_1(X)$ stands for the open sets of $X$, $\Pi^0_1(X)$ for the closed sets, $\Sigma_2^0(X)$ for the $F_\sigma$-sets, etc.; 
we denote by $\Delta^0_\alpha(X):=\Sigma^0_\alpha(X) \cap \Pi^0_\alpha(X)$ the ambiguous classes; 
also, if $X$ is a Polish space, 
$\Sigma^1_1(X)$ stands for the analytic sets, $\Pi^1_1(X)$ for the coanalytic sets, etc., see e.g. 
\cite[Section 11.B]{K} or \cite[Section 3.6]{S}. Again, we suppress the reference to the underlying space $X$ if it is clear from the context. 

Recall that an ideal $\mathcal{I}$ is a \emph{P-ideal} if it is $\sigma$-directed modulo finite sets, that is, for every sequence $(A_n)$ with values in $\mathcal{I}$ there exists $A \in \mathcal{I}$ such that $A_n\setminus A$ is finite for all $n \in \omega$. Important examples include the $\Sigma^1_1$ P-ideals, which are known to be necessarily $\Pi^0_3$. 
A topological space $X$ is \emph{discrete} if it contains only isolated points.

\begin{thm}\label{thm:oldthm}
Let $X$ be a nondiscrete first countable Hausdorff space. Then\textup{:} 
\begin{enumerate}[label={\rm (\roman{*})}]
\item \label{old:prelim1} $\mathscr{L}(\mathcal{I})\subseteq \Pi^0_1$ if and only if $\mathcal{I}$ is a $\Sigma_2^0$ ideal, provided that $\mathcal{I}$ is a $\Sigma^1_1$ P-ideal.
\item \label{old:prelim2} $\mathscr{L}(\mathcal{I})\subseteq \Sigma^0_2$, provided that $\mathcal{I}$ is a $\Sigma^1_1$ P-ideal.
\end{enumerate}
If, in addition, all closed subsets of $X$ are separable, then\textup{:}
\begin{enumerate}[label={\rm (\roman{*})}]
\setcounter{enumi}{2}
\item \label{old:prelim3} $\mathscr{L}(\mathcal{I})=\Pi_1^0$, provided that $\mathcal{I}$ is a $\Sigma_2^0$ ideal.
\item \label{old:prelim4} $\mathscr{L}(\mathcal{I})=\Sigma_2^0$, provided that $\mathcal{I}$ is $\Sigma_1^1$ P-ideal which is not $\Sigma_2^0$.
\end{enumerate}
\end{thm}
\begin{proof}
It follows by \cite[Theorems 2.2, 2.5, 2.7, and 3.4]{BL}. 
\end{proof}

To state the next result, we recall some further definitions. 
An ideal $\mathcal{I}$ has the \emph{hereditary Baire property} if the restriction $\mathcal{I}\upharpoonright A:=\{S\cap A: S \in \mathcal{I}\}$ has the Baire property for every $A \in \mathcal{I}^+$. Note that all analytic ideals have the hereditary Baire property: indeed, the proof goes verbatim as in 
\cite[Theorem 3.13]{MR4358610}, considering that analytic sets are closed under continuous preimages, and that they have the Baire property, see e.g. \cite[Proposition 4.1.2 and Theorem 4.3.2]{S}. In addition, there exist ideals with the Baire property but without the hereditary Baire property, see e.g. \cite[Proposition 2.1]{MR3624786}.

Also, an ideal $\mathcal{I}$ on $\omega$ is said to be a $P^+$\emph{-ideal} if, for every decreasing sequence $(A_n)$ with values in $\mathcal{I}^+$, there exists $A \in \mathcal{I}^+$ such that $A\setminus A_n$ is finite for all $n \in \omega$. It is known that all $F_\sigma$-ideals are $P^+$-ideals, see 
\cite{MR748847} and 
\cite[Observation 2.2]{MR2861027}. 
We remark that $P^+$-ideal may have arbitrarily high Borel complexity, as it has been proved in \cite[p. 2031]{MR3692233} and \cite[Example 2.6]{MR2861027}, 
cf. also \cite{MR4330212}. 

Moreover, an ideal $\mathcal{I}$ is called a \emph{Farah ideal} if there exists a sequence $(K_n)$ of hereditary compact subsets in $\mathcal{P}(\omega)$ such that $S \in \mathcal{I}$ if and only if for all $n \in \omega$ there exists $k \in \omega$ such that $S\setminus [0,k] \in K_n$, see \cite{MR2048515, MR2777744, MR2849045}. It is known that all analytic P-ideals are Farah and that all Farah ideals are $F_{\sigma\delta}$. On the other hand, it is still unknown whether 
the converse holds, namely, all $F_{\sigma\delta}$ ideals are Farah ideals, 
see \cite{MR1711328} and \cite[p. 60]{MR2777744}. 

\begin{thm}\label{thm:xi}
Let $X$ be a first countable Hausdorff space. Then\textup{:}
\begin{enumerate}[label={\rm (\roman{*})}]
\item \label{item:1xi} $\mathscr{L}(\mathcal{I}) \subseteq \Pi^0_1$, provided that $\mathcal{I}$ is $P^+$-ideal.
\item \label{item:2xi} $\mathscr{L}(\mathcal{I}) \subseteq \Pi^0_1$ if and only if $\mathcal{I}$ is a $P^+$-ideal, provided that $\mathcal{I}$ has the hereditary Baire property and $X$ is nondiscrete metrizable. 
\item \label{item:3xi} $\mathscr{L}(\mathcal{I}) \subseteq \Sigma^0_2$, provided that $\mathcal{I}$ is a Farah ideal.
\item \label{item:3xiB} $\mathscr{L}(\mathcal{I})=\{\{\eta\}: \eta \in X\}$ if and only if $\mathcal{I}$ is maximal, provided that $|X|\ge 2\textup{.}$
\end{enumerate}
If, in addition, $X$ is second countable, then\textup{:}
\begin{enumerate}[label={\rm (\roman{*})}]
\setcounter{enumi}{4}
\item \label{item:4xi} $\Pi^0_1 \subseteq \mathscr{L}(\mathcal{I})$, provided that $\mathcal{I}$ has the hereditary Baire property.
\item \label{item:5xi} $\Sigma^0_2 \subseteq \mathscr{L}(\mathcal{I})$, provided that $\mathcal{I}$ has the hereditary Baire property and is not a $P^+$-ideal. 
\end{enumerate}
\end{thm}
\begin{proof}
It follows by \cite[Proposition 1.4, Corollary 3.9, and Theorem 4.5]{Xi} and \cite[Corollary 2.5 and Theorems 2.4, 2.8, and 2.10]{Xi2}. (We added the \textsc{only if} part of item \ref{item:3xiB}, which is straighforward.)

%
\end{proof}

In what follows, we divide our main results into four sections. First, we show in Section \ref{sec:large} that $\mathscr{L}(\mathcal{I})$ can be equal to families of arbitrarily high Borel complexity. 
Also, we prove 
that we cannot have the equality $\mathscr{L}(\mathcal{I})=\Pi^0_2$: more precisely, $\mathscr{L}(\mathcal{I})\subseteq \Pi^0_2$ if and only if $\mathscr{L}(\mathcal{I})\subseteq \Pi^0_1$. 
Then, we show that, if $X$ is a first countable space and $\mathcal{I}$ is a simply coanalytic ideal (see Section \ref{sec:sanalytic} for details), then every $\Lambda_{\bm{x}}(\mathcal{I})$ is simply analytic. 
Lastly, we study structural and topological properties of \textquotedblleft smallest\textquotedblright\, ideals $\mathcal{I}$ such that $\mathscr{L}(\mathcal{I})$ contains a given subsets of $X$.

\section{Large and small Borel complexities} \label{sec:large}
Our first main result 
computes explicitly some families $\mathscr{L}(\mathcal{I})$, proving that they may attain arbitrarily large Borel complexity (Theorem \ref{thm:highBorel} below). This is somehow related to \cite[Question 3.11]{Xi}, which asks about the existence of a Borel ideal $\mathcal{I}$ such that $\mathscr{L}(\mathcal{I})$ contains sets with large Borel complexities.

For, 
given (possibly nonadmissible) ideals $\mathcal{I}$ and $\mathcal{J}$ on two countably infinite sets $Z$ and $W$, respectively, we define their \emph{Fubini product} by 
$$
\mathcal{I}\times \mathcal{J}:=\{S\subseteq Z\times W\colon \{n \in Z: \{k\in W: (n,k)\in S\}\notin \mathcal{J}\}\in \mathcal{I}\}\},
$$
which is an ideal on the countably infinite set $Z\times W$, see e.g. \cite[Chapter 1]{MR1711328}. Hence, 
recursively, 
$\mathrm{Fin}^{\alpha}:=\mathrm{Fin}\times \mathrm{Fin}^{\alpha-1}$ for all integers $\alpha\ge 2$ is an ideal on $\omega^\alpha$. 

\begin{thm}\label{thm:baseinductionemptysetGENERAL}
Let $\mathcal{I}, \mathcal{J}$ be \textup{(}possibly nonadmissible\textup{)} ideals on $\omega$ such that $\mathcal{I}\times \mathcal{J}$ is an admissible ideal on $\omega^2$.
Fix also a bijection $h: \omega^2\to \omega$ and let $\bm{x}=(x_{n})$ be a sequence with values in a first countable space $X$. 
For each $n \in \omega$, define the sequence $\bm{x}^{(n)}=(x^{(n)}_k: k \in \omega)$ by 
$x^{(n)}_k:=x_{h(n,k)}$ for all $k \in \omega$. 
Then 
\begin{equation}\label{eq:claimedequalityeee}
\Lambda_{\bm{x}}(h[\,\mathcal{I} \times \mathcal{J}])=\left\{\eta \in X: \left\{n \in \omega: \eta \in \Lambda_{\bm{x}^{(n)}}(\mathcal{J})\right\} \notin \mathcal{I}\right\},
\end{equation}
 where $h[\,\mathcal{I} \times \mathcal{J}]$ stands for the family $\{h[S]: S \in \mathcal{I} \times \mathcal{J}\}$.
\end{thm}
\begin{proof}
Let $A$ and $B$ be the left and right hand side of \eqref{eq:claimedequalityeee}, respectively. 

\medskip

\textsc{Inclusion $A\subseteq B$.} The inclusion is clear if $A=\emptyset$. Otherwise fix a point $\eta \in A$. 
Hence there exists a set $S\subseteq \omega^2$ such that $S\notin \mathcal{I}\times \mathcal{J}$ and the subsequence $(x_{h(s)}: s \in S)$ is convergent to $\eta$. 
By the definition of Fubini product $\mathcal{I}\times \mathcal{J}$, 
$$
N:=\{n \in \omega: K_n\in \mathcal{J}^+\} \in \mathcal{I}^+, 
\,\, \text{ where }\,\,
K_n:=\{k\in \omega: (n,k) \in S\}.
$$
At this point, for each $n \in N$, the subsequence $(x^{(n)}_{k}: k\in K_n)$ is convergent to $\eta$ and $K_n \in \mathcal{J}^+$. Since $N \in \mathcal{I}^+$, we obtain that $\eta \in B$. 

\medskip

\textsc{Inclusion $B\subseteq A$.} 
The inclusion is clear if $B=\emptyset$. Otherwise fix a point $\eta \in B$ and let $(U_n)$ be a decreasing local base of neighborhoods at $\eta$. Hence there exists $N \in \mathcal{I}^+$ such that $\eta$ is a $\mathcal{J}$-limit point of $\bm{x}^{(n)}$ for each $n \in N$, let us say $\lim_{k \in K_n}x^{(n)}_k=\eta$ for some $K_n \in \mathcal{J}^+$. Upon removing finitely many elements, we can suppose without loss of generality that 
$$
\forall n \in N, \forall k \in K_n, \quad 
x^{(n)}_k \in U_n.
$$
Now, set $S:=\{(n,k)\in \omega^2: n \in N, k \in K_n\}$ and note that $S\notin \mathcal{I}\times \mathcal{J}$. 
It follows that the subsequence $(x_{h(s)}: s \in S)$ is convergent to $\eta$: indeed there are only finitely many elements of the subsequence outside each $U_n$. Therefore $\eta \in A$, which concludes the proof. 
\end{proof}
\begin{rmk}\label{rmk:notnecessarilyfirstcountablee}
The above result holds for every topological space $X$ if $\mathcal{I}=\{\emptyset\}$. Indeed, in the second part of the proof it is enough to let $N$ be a singleton.
\end{rmk}

It is worth noting that if $\mathcal{I}$ is an ideal on a countably infinite set $Z$, $\phi: Z\to \omega$ is a bijection, and $\bm{x}\in X^Z$ is a $Z$-indexed sequence with values in a topological space $X$, then $\phi[\,\mathcal{I}]:=\{\phi[S]: S \in \mathcal{I}\}$ is an ideal on $\omega$ and 
$
\Lambda_{\bm{x}}(\mathcal{I})=\Lambda_{\bm{y}}(\phi[\,\mathcal{I}]), 
$  
where $\bm{y} \in X^\omega$ is the sequence defined by $y_n:=x_{\phi^{-1}(n)}$ for all $n \in \omega$. Hence 
we may use interchangeably $\mathscr{L}(\mathcal{I})$ or $\mathscr{L}(\phi[\,\mathcal{I}])$. 
In particular, Theorem \ref{thm:baseinductionemptysetGENERAL} allows to compute explicitly families of the type $\mathscr{L}(\mathcal{I}\times \mathcal{J})$. 
For, we state two consequence of Theorem \ref{thm:baseinductionemptysetGENERAL}: 

\begin{cor}\label{cor:emptysetfubini}
Let $X$ be a topological space and $\mathcal{I}$ be an ideal on $\omega$. Then 
$$
\mathscr{L}(\emptyset \times \mathcal{I})=\left\{\bigcup\nolimits_nA_n: A_0,A_1,\ldots \in \mathscr{L}(\mathcal{I})\right\}.
$$
\end{cor}
\begin{proof}
It follows by Theorem \ref{thm:baseinductionemptysetGENERAL} and Remark \ref{rmk:notnecessarilyfirstcountablee} that 
\begin{displaymath}
\begin{split}
\mathscr{L}(\emptyset \times \mathcal{I})
&=\left\{\Lambda_{\bm{x}}(\emptyset \times \mathcal{I}): {\bm{x}} \in X^\omega\right\}\\
&=\left\{\bigcup\nolimits_{n}\Lambda_{{\bm{x}}^{(n)}}(\mathcal{I}): {\bm{x}} \in X^\omega\right\}\\
&=\left\{\bigcup\nolimits_{n}\Lambda_{{\bm{x}}^{(n)}}(\mathcal{I}): {\bm{x}}^{(0)}, {\bm{x}}^{(1)}, \ldots \in X^\omega\right\}\\
&=\left\{\bigcup\nolimits_nA_n: A_0,A_1,\ldots \in \mathscr{L}(\mathcal{I})\right\},
\end{split}
\end{displaymath}
completing the proof. 
\end{proof}

\begin{cor}\label{cor:emptysetfubini2}
Let $X$ be a first countable space and $\mathcal{I}$ be an ideal on $\omega$. Then 
$$
\mathscr{L}(\mathrm{Fin}\times \mathcal{I})=\left\{\limsup\nolimits_nA_n: A_0,A_1,\ldots \in \mathscr{L}(\mathcal{I})\right\}.
$$
\end{cor}
\begin{proof}
It follows by Theorem \ref{thm:baseinductionemptysetGENERAL} that 
\begin{displaymath}
\begin{split}
\mathscr{L}(\mathrm{Fin} \times \mathcal{I})
&=\left\{\Lambda_{\bm{x}}(\mathrm{Fin} \times \mathcal{I}): {\bm{x}} \in X^\omega\right\}\\
&=\left\{\left\{\eta \in X: \exists^\infty n \in \omega, \eta \in \Lambda_{{\bm{x}}^{(n)}}(\mathcal{I})\right\}: {\bm{x}} \in X^\omega\right\}\\
&=\left\{\bigcap\nolimits_n\bigcup\nolimits_{k\ge n}\Lambda_{{\bm{x}}^{(n)}}(\mathcal{I}): {\bm{x}}^{(0)}, {\bm{x}}^{(1)}, \ldots \in X^\omega\right\}\\
&=\left\{\limsup\nolimits_nA_n: A_0,A_1,\ldots \in \mathscr{L}(\mathcal{I})\right\},
\end{split}
\end{displaymath}
completing the proof. 
\end{proof}

At this point, recall that $\mathrm{Fin}$ is a $F_\sigma$-ideal, and that $\emptyset\times \mathrm{Fin}$ is an analytic $P$-ideal which is not $F_\sigma$, see \cite[Example 1.2.3]{MR1711328} and \cite[Exercise 23.1]{K}. Thus, it follows by Theorem \ref{thm:oldthm} that
\begin{equation}\label{eq:basicrelationshipL}
\mathscr{L}(\mathrm{Fin})= \Pi^0_{1}
\quad \text{ and }\quad 
\mathscr{L}(\emptyset \times \mathrm{Fin})= \Sigma^0_{2}.
\end{equation}

With the above premises, we are able to extend 
\eqref{eq:basicrelationshipL} 
to certain ideals with large Borel complexity families $\mathscr{L}(\mathcal{I})$. 
\begin{thm}\label{thm:highBorel}
Let $X$ be a complete metrizable space. Then, for each positive integer $\alpha$, we have\textup{:} 
 
\begin{enumerate}[label={\rm (\roman{*})}]
\item \label{item:1mainlimitfamilies} $\mathrm{Fin}^\alpha$ is a $\Sigma^0_{2\alpha}$-ideal and $\mathscr{L}(\mathrm{Fin}^\alpha)= \Pi^0_{2\alpha-1}$.
\item \label{item:2mainlimitfamilies} $\emptyset \times \mathrm{Fin}^\alpha$ is a $\Pi^0_{2\alpha+1}$-ideal and $\mathscr{L}(\emptyset \times \mathrm{Fin}^\alpha)= \Sigma^0_{2\alpha}$. 
\end{enumerate}
\end{thm}
\begin{proof}
\ref{item:1mainlimitfamilies} 
The complexity of $\mathrm{Fin}^\alpha$ is obtained applying recursively \cite[Proposition 1.6.16]{Mezathesis}, while the computation of $\mathscr{L}(\mathrm{Fin}^\alpha)$ is obtained putting together the base case \eqref{eq:basicrelationshipL}, Corollary \ref{cor:emptysetfubini2}, and \cite[Exercise 23.5(i)]{K}. 

The proof of \ref{item:2mainlimitfamilies} goes similarly, replacing Corollary \ref{cor:emptysetfubini2} with Corollary \ref{cor:emptysetfubini}. 
\end{proof}

In particular, if $X$ is a complete metrizable space, $\mathscr{L}_X(\mathrm{Fin}^2)=\Pi^0_3$. This provides a generalization of \cite[Example 4.2]{BL}, where it is proved constructively that there exists a \emph{real} sequence $\bm{x}$ such that $\Lambda_{\bm{x}}(\mathrm{Fin}^2)$ is equal to $[0,1]\setminus \mathbb{Q}$ (note that the latter is not a $F_\sigma$-set, hence $\mathscr{L}_{\mathbb{R}}(\mathrm{Fin}^2)\cap (\Pi^0_2 \setminus \Sigma^0_2)$ is nonempty). 

\medskip

Our second main result deals with ideals with small topological complexity. Suppose that $X$ is a first countable space, and recall that 
\begin{equation}\label{eq:inclusionclosedsets}
\mathscr{L}(\mathcal{I})\subseteq \Pi^0_1,
\end{equation}
provided $\mathcal{I}$ is a $\Sigma^0_2$-ideal, see \cite[Theorem 2.3]{BL}; note that the Hausdorffness hypothesis is not needed here. In the next result, we are going to show that the same conclusion holds if $\mathcal{I}$ is a $\Sigma^0_3$-ideal.
\begin{thm}\label{thm:Gdeltasigmaideals}
Let $X$ be a first countable space, and suppose that $\mathcal{I}$ is a $\Sigma^0_3$-ideal. Then inclusion \eqref{eq:inclusionclosedsets} holds. 
\end{thm}

It is worth noting that Theorem \ref{thm:Gdeltasigmaideals} is not a consequence of the former result as it really includes new cases: indeed, as remarked also by Solecki in \cite[p. 345]{MR1416872}, there exists a $\Delta^0_3$-ideal on $\omega$ (hence, both $\Sigma^0_3$ and $\Pi^0_3$) which is neither $\Sigma^0_2$ nor $\Pi^0_2$, see \cite{MR1321463}; cf. also \cite{Calb1, Calb2} and \cite[Proposition 1.2.1]{Mezathesis}.  

In addition, Theorem \ref{thm:Gdeltasigmaideals} allows us to prove a generalization of the folklore result that every $\Sigma^0_2$-ideal is a $P^+$-ideal, see \cite{MR748847}. For a different proof of the second part, see also \cite[Proposition 10.1]{MR4448270}. 
\begin{cor}\label{cor:justP+}
Let $\mathcal{I}$ be a $\Sigma^0_3$-ideal on $\omega$. Then $\mathcal{I}$ is a $P^+$-ideal, and there exists a $\Sigma^0_2$-ideal $\mathcal{J}$ such that $\mathcal{I}\subseteq \mathcal{J}$. 
\end{cor}
\begin{proof}
The first part follows putting together Theorem \ref{thm:Gdeltasigmaideals} and Theorem \ref{thm:xi}.\ref{item:2xi} (note that $\mathcal{I}$ is Borel, hence with the hereditary Baire property). The second part follows by the known fact that a Borel ideal on $\omega$ is contained in a $\Sigma^0_2$-ideal if and only if it is contained in a $P^+$-ideal, see \cite{MR748847}. 
\end{proof}

At this point, we divide the proof of Theorem \ref{thm:Gdeltasigmaideals} into two intermediate steps. 
To this aim, we recall that properties of ideals can be often expressed by finding \emph{critical} ideals with respect to some preorder, cf. e.g. the survey \cite{MR2777744}. 
To this aim, let $\mathcal{I}$ and $\mathcal{J}$ be two ideals on two countably infinite sets $Z$ and $W$, respectively. Then we say that $\mathcal{I}$ \emph{is below} $\mathcal{J}$ \emph{in the Rudin--Blass ordering}, shortened as 
$$
\mathcal{I} \le_{\mathrm{RB}} \mathcal{J},
$$ 
if there is a finite-to-one map $\phi: W\to Z$ such that $S \in \mathcal{I}$ if and only if $\phi^{-1}[S] \in \mathcal{J}$ for all subsets $S\subseteq Z$. 
The restriction of these orderings to maximal ideals $\mathcal{I}$, and the Borel cardinality of the quotients $\mathcal{P}(Z)/\mathcal{I}$ have been extensively studied, see e.g. \cite{MR0396267, MR1476761, MR1617463} and references therein.  

\begin{lem}\label{lem:step1Gdeltasigma}
Let $X$ be a first countable Hausdorff space and $\mathcal{I}$ be an ideal on $\omega$ with the hereditary Baire property. Suppose that inclusion \eqref{eq:inclusionclosedsets} does not hold. Then $\emptyset \times \mathrm{Fin}\le_{\mathrm{RB}} \mathcal{I}$. 
\end{lem}
\begin{proof}
Since inclusion \eqref{eq:inclusionclosedsets} fails and $X$ is first countable, there exists a sequence $\bm{x}$ such that $S:=\Lambda_{\bm{x}}(\mathcal{I})$ is not closed, hence not sequentially closed. Therefore there exists a sequence $\bm{y}$ taking values in $S$ which is convergent to some limit $\eta \in X\setminus S$. Since $X$ is Hausdorff, we may suppose without loss of generality that $y_n \neq y_m$ for all distinct $n,m \in \omega$. Now, for each $n \in \omega$, there exists $A_n \in \mathcal{I}^+$ such that $\lim_{k \in A_n}x_k=y_n$. Define $B_0:=A_0\cup (\omega\setminus \bigcup_nA_n)$ and, recursively, $B_{n+1}:=A_{n+1}\setminus \bigcup_{k\le n}B_k$ for all $n\in \omega$. Hence $\{B_n: n\in \omega\}$ is a partition of $\omega$ into $\mathcal{I}$-positive sets, for each $n \in \omega$, the restriction $\mathcal{I}\upharpoonright B_n$ is an ideal on $B_n$ with the Baire property. 

It follows by Talagrand's characterization of meager ideals that $\mathrm{Fin} \le_{\mathrm{RB}} \mathcal{I}\upharpoonright B_n$ for each $n \in \omega$. More explicitly, for each $n \in \omega$ there exists a finite-to-one map $\phi_n: B_n \to \omega$ such that 
\begin{equation}\label{eq:RB1}
\forall W \subseteq \omega, \quad 
\phi_n^{-1}[W] \in \mathcal{I} 
\,\,\,\text{ if and only if }\,\,\,
W \in \mathrm{Fin},
\end{equation}
see \cite[Theorem 2.1]{MR0579439}; cf. also \cite{MR4566746} for further characterizations of meager ideals based of $\mathcal{I}$-limit points of sequences. 
Define the map   
$\phi: \omega\to \omega^2$ by 
$$
\forall n\in \omega, \forall k \in B_n, \quad 
\phi(k)=(n,\phi_n(k)).
$$

We claim that $\phi$ is a witnessing function for $\emptyset\times \mathrm{Fin}\le_{\mathrm{RB}}\mathcal{I}$. 
For, suppose that $W\subseteq \omega^2$ belongs to $\emptyset \times \mathrm{Fin}$. Then 
\begin{equation}\label{eq:RB2}
\phi^{-1}[W]=\bigcup_{n \in \omega}\phi_n^{-1}[\{k \in \omega: (n,k) \in W\}]
\end{equation}
Since each $\phi_n$ is finite-to-one, $\phi^{-1}[W]$ has finite intersection with each $B_n$. Then either $\phi^{-1}[W]$ is finite or the subsequence $(x_n: n \in \phi^{-1}[W])$ is convergent to $\eta$, while $\eta$ is not an $\mathcal{I}$-limit point of $\bm{x}$. Hence, in both cases, $\phi^{-1}[W] \in \mathcal{I}$. Conversely, suppose that $W$ does not belong to $\emptyset \times \mathrm{Fin}$, so that there exists $n_0 \in \omega$ such that $\{k \in \omega: (n_0,k) \in W\}\notin \mathrm{Fin}$. It follows by \eqref{eq:RB1} and \eqref{eq:RB2} that $\phi^{-1}[W]$ contains 
$$
\phi_{n_0}^{-1}[\{k \in \omega: (n_0,k) \in W\}] \in \mathcal{I}^+. 
$$
Therefore $\phi^{-1}[W]\in \mathcal{I}$ if and only if $W \in \emptyset \times \mathrm{Fin}$ for each $W\subseteq \omega^2$. 
\end{proof}

\begin{cor}\label{cor:RB}
Under the same hypotheses of Lemma \ref{lem:step1Gdeltasigma}, $\Sigma^0_2 \subseteq \mathscr{L}(\mathcal{I})$. 
\end{cor}
\begin{proof}
Thanks to Lemma \ref{lem:step1Gdeltasigma} and \cite[Proposition 3.8]{Xi}, we have $\mathscr{L}(\emptyset \times \mathrm{Fin}) \subseteq \mathscr{L}(\mathcal{I})$. The claim follows by 
Equation \eqref{eq:basicrelationshipL}. 
\end{proof}

For the next intermediate result, given topological spaces $X,Y$ and subsets $A\subseteq X$ and $B\subseteq Y$, we say that $A$ \emph{is Wadge reducible to} $B$, shortened as 
$$
A \le_{\mathrm{W}} B,
$$
if there exists a continuous map $\Phi: X \to Y$ such that $\Phi^{-1}[B]=A$ (or, equivalently, $x \in A$ if and only if $\Phi(x) \in B$ for all $x \in X$), see e.g. \cite[Definition 21.13]{K}. If, in addition, $X$ and $Y$ are Polish spaces with $X$ zero-dimensional, then $B$ is said to be $\Pi^0_3$\emph{-hard} if $A \le_{\mathrm{W}} B$ for some $A \in \Pi^0_3(X)$. Lastly, if $B$ is a $\Pi^0_3(Y)$ set which is also $\Pi^0_3$-hard, then it is called $\Pi^0_3$\emph{-complete}. (Analogous definitions can be given for other classes of sets in Polish spaces, see \cite[Definition 22.9]{K}.)

\begin{lem}\label{lem:step2Gdeltasigma}
Let $\mathcal{I}$ be an ideal on $\omega$ such that $\emptyset \times \mathrm{Fin}\le_{\mathrm{RB}} \mathcal{I}$. Then $\mathcal{I}$ is $\Pi^0_3$-hard.
\end{lem}
\begin{proof}
First, recall that $\emptyset \times \mathrm{Fin}$ is a $\Pi^0_3$-complete subset of $\mathcal{P}(\omega^2)$, see e.g. \cite[Exercise 23.1]{K}. Hence, to complete the proof, it is sufficient to show that
$$
\emptyset\times \mathrm{Fin} \le_{\mathrm{W}} \mathcal{I}. 
$$

By hypothesis, there exists a finite-to-one function $\phi: \omega \to \omega^2$ such that $S \in \emptyset \times \mathrm{Fin}$ if and only if $\phi^{-1}[S] \in \mathcal{I}$. 
Now, define the map 
$\Phi: \mathcal{P}(\omega^2) \to \mathcal{P}(\omega)$ by $\Phi(S):=\phi^{-1}[S]$, so that $S \in \emptyset \times \mathrm{Fin}$ if and only if $\Phi(S) \in \mathcal{I}$. 
Hence, we only need to show that $\Phi$ is continuous. For, fix $n \in \omega$ and define $S:=\phi[\{0,\ldots,n\}]$. 
It follows that, for all $A,B\subseteq \omega^2$ with $A\cap S=B\cap S$, 
\begin{displaymath}
\begin{split}
\Phi(A) \cap [0,n]
&
=\{k \in \omega: \phi(k) \in A \text{ and }k\le n\}\\
&=\{k \in \omega: \phi(k) \in A \cap S\}
=\Phi(B) \cap [0,n],
\end{split}
\end{displaymath}
which proves the continuity of $\Phi$. 
\end{proof}

\begin{proof}
[Proof of Theorem \ref{thm:Gdeltasigmaideals}]
  Suppose that $X$ is a first countable space and $\mathcal{I}$ is an ideal on $\omega$ such that inclusion \eqref{eq:inclusionclosedsets} fails. 
  We claim that $\mathcal{I}$ is not a $\Sigma^0_3$-ideal. 

If $\mathcal{I}$ is not a Borel ideal, then the claim is trivial. Hence, let us suppose hereafter that $\mathcal{I}$ is Borel. In particular, $\mathcal{I}$ has the hereditary Baire property. 
At this point, it follows by Lemma \ref{lem:step1Gdeltasigma} and Lemma \ref{lem:step2Gdeltasigma} that $\mathcal{I}$ is a $\Pi^0_3$-hard. To sum up, $\mathcal{I}$ is a Borel subset of a zero-dimensional Polish space and it is $\Pi^0_3$-hard. We conclude by \cite[Therem 22.10]{K} that the $\mathcal{I}$ is not a $\Sigma^0_3$-ideal. 
\end{proof}

\section{Simply analyticity}\label{sec:sanalytic}


In \cite[Proposition 4.1]{BL}, the first and last-named authors proved the following: 
\begin{prop}\label{prop:original}
Let $\bm{x}$ be a sequence taking values in a Hausdorff regular first countable space $X$. Let also $\mathcal{I}$ be a coanalytic ideal on $\omega$. Then $\Lambda_{\bm{x}}(\mathcal{I})$ is analytic. 
\end{prop}

However, the classical definition of analytic sets $A\subseteq X$ as projections of Borel subsets of $X\times X$ is usually considered in \emph{Polish spaces} $X$, cf. e.g. \cite[Chapter 4]{S} or \cite[Chapter 14]{K}. 
In addition, the above result has been re-proved in \cite[Theorem 4.1]{Xi2} for zero-dimensional Polish spaces $X$, so that 
$$
\mathscr{L}(\mathcal{I})\subseteq \Sigma^1_1.
$$
whenever $\mathcal{I}$ is a $\Pi^1_1$-ideal (notice that the above notation is meaningful). 

In this Section, our aim is to reformulate and clarify the statement and the proof of Proposition \ref{prop:original}, extending in turn both the latter and the special case treated in \cite[Theorem 4.1]{Xi2}.

Note that there are several papers in which the notion of analytic set is adapted to more general topological spaces, see e.g. \cite{HJR, R}. The theory of the latter extension, called $K$-analytic sets, is important and nontrivial, and it does not proceed verbatim as in the classical one of analytic sets. 

For our purposes, we generalize in a straight way one of possible definitions of analytic sets to the case of arbitrary topological spaces. Hereafter, $\pi_X$ stands for the usual projection on $X$. 
\begin{defi}\label{def:sanalytic}
Let $X$ be a topological space. A subset $A\subseteq X$ is said to be \emph{simply analytic}, shortened as \emph{s-analytic}, if there exists an uncountable Polish space $Y$ and a Borel subset $B\subseteq X\times Y$ such that $A=\pi_X[B]$. 
\end{defi}

Note that Definition \ref{def:sanalytic} is coherent with the classical notion of analytic set. 
Indeed, it is well known that a subset $A$ of a Polish space $X$ is analytic if and only if there exists an uncountable Polish space $Y$ and a Borel $B\subseteq X\times Y$ such that $A=\pi_X[B]$, cf. \cite[Proposition 4.1.1]{S}. At this point, let $Y_1,Y_2$ be two uncountable Polish spaces. Then there exists a Borel isomorphism $h: Y_2\to Y_1$, as it follows by \cite[Theorem 3.3.13]{S}. Let $A:=\pi_X[B_1]$ be a s-analytic subset for a Borel set $B_1\subseteq X\times Y_1$. Then $B_2:=\{(x,y) \in X\times Y_2: (x,h(y)) \in B_1\}$ is a Borel subset of $X\times Y_2$ by Lemma \ref{lem:compositionBorel} below and, clearly, $\pi_X[B_2]=A$. Therefore, in Definition \ref{def:sanalytic} one may assume without loss of generality that $Y=\omega^\omega$. 

Lastly, a subset $A\subseteq X$ is said to be \emph{s-coanalytic} if $A^c:=X\setminus A$ is s-analytic. 
Observe that every Borel set in $X$ is both s-analytic and s-coanalytic.

These observations enable us to consider Proposition \ref{prop:original} valid as it was stated in the original version in \cite{BL}. They allow, in addition, to remove in its statement the hypothesis of regularity (and also the property of being Hausdorff; however, the latter one has been used in \cite{BL} only because a first countable space is Hausdorff if and only if every sequence has at most one limit):
\begin{thm}\label{prop:Polish}
Let ${\bm{x}}$ be a sequence taking values in a 
first countable space $X$. Let also $\mathcal{I}$ be a s-coanalytic ideal on $\omega$. Then $\Lambda_{\bm{x}}(\mathcal{I})$ is s-analytic. 
\end{thm} 

For the proof of Theorem \ref{prop:Polish}, we will need some intermediate lemmas. 

\begin{lem}\label{lem:compositionBorel}
Let $f: X\to Z$ and $g: Y\to W$ be Borel functions, where $X,Y,Z,W$ are topological spaces. 
Then the map $h: X\times Y \to Z\times W$ defined by 
$$
\forall (x,y) \in X\times Y, \quad h(x,y):=(f(x),g(y))
$$
is Borel. 
\end{lem}
\begin{proof}
Let $U\subseteq Z$ and $V\subseteq W$ be arbitrary open sets. It is enough to show that $h^{-1}[U\times V]$ is a Borel subset of $X\times Y$. The latter set is equal to $f^{-1}[U]\times g^{-1}[V]$, which belongs to $\mathscr{B}(X)\otimes \mathscr{B}(Y)$. The claim follows by the inclusion $\mathscr{B}(X)\otimes \mathscr{B}(Y)\subseteq \mathscr{B}(X\times Y)$, see e.g. \cite[Proposition 251E]{MR2462280}. (For the converse inclusion of the latter, which does not hold for every $X$ and $Y$, see \cite{MR1343692}.)
\end{proof}

\begin{lem} \label{lem:fafa}
Let $f\colon X\to Y$ be a Borel function, where $X$ is a topological space and $Y$ is an uncountable Polish space. If $A\subseteq Y$ is an analytic set, then $f^{-1}[A]$ is a s-analytic subset of $X$.
\end{lem}
\begin{proof}
Since $A\subseteq Y$ is analytic, there exists a Borel subset $B\subseteq Y\times Y$ such that $A$ is the projection on the first coordinate of $B$. 
At this point, define 
$C:=\left\{(x,y) \in X\times Y: (f(x),y) \in B\right\}$. 
Then $C$ is Borel set by Lemma \ref{lem:compositionBorel} and $\pi_X[C]=f^{-1}[A]$. Therefore $f^{-1}[A]$ is a s-analytic set, which concludes the proof. 
\end{proof}

\begin{lem}\label{lem:projanalytic}
Let $A\subseteq  X\times Y$ be a s-analytic set, where $X$ is a topological space and $Y$ is an uncountable Polish space. Then $C:=\pi_X[A]$ is a s-analytic set.
\end{lem}
\begin{proof}
Let $Z$ be an uncountable Polish space and $B\subseteq X\times Y\times Z$ be a Borel set such that $A=\pi_{X\times Y}[B]$. 
Now, observe that $C=\pi_X[B]$ and obviously $Y\times Z$ is an uncountable Polish space.
\end{proof}

\begin{lem}\label{lem:countableintersection}
Let $(A_n)$ be a sequence of s-analytic subsets of a topological space $X$. 
Then both $\bigcup\nolimits_n A_n$ and $\bigcap\nolimits_n A_n$ are s-analytic. 
\end{lem}
\begin{proof}
The proof proceeds verbatim as in \cite[Proposition 4.1.2]{S}.
\end{proof}



We are ready for the proof of Theorem \ref{prop:Polish}. 
\begin{proof}
[Proof of Theorem \ref{prop:Polish}]
Note that $\mathcal{I}^+$ 
is s-analytic and let $\mathcal{N}$ be the set of strictly increasing sequences $(n_k)$ of nonnegative integers. 
Since $\mathcal{N}$ can be regarded as a closed subset of $\omega^\omega$, it is a Polish space by Alexandrov's theorem, see e.g. \cite[Theorem 2.2.1]{S}. 
For each $\eta\in X$, 
fix a decreasing local base $\left(U_{\eta,m}: m \in \omega\right)$ of open neighbourhoods of $\eta$. 
Then $\eta$ is an $\mathcal{I}$-limit point of $\bm{x}$ if and only if there exists a sequence $(n_k) \in \mathcal{N}$ such that $\{n_k: k \in \omega\} \in \mathcal{I}^+$ and 
$\{k \in \omega: x_{n_k}\notin U_{\eta,m}\} \in \mathrm{Fin}$ for all $m \in \omega$. 

At this point, define the continuous function 
$$
\psi\colon\mathcal{N}\to\{0,1\}^\omega: (n_k) \mapsto \chi_{\{n_k\colon k\in\omega\}},
$$
where $\chi_S$ stands for the characteristic function of a set $S\subseteq\omega$. 
In addition, for each $m \in \omega$, define the function $\zeta_m\colon \mathcal{N}\times X\to \{0,1\}^\omega$ by
$$
\forall (n_k) \in \mathcal{N}, \forall \eta \in X, 
\quad 
\zeta_m\left((n_k), \eta\right):=\chi_{\{\,t \in \omega:\, x_{n_t}\notin U_{\eta,m}\}}
$$
Identifying each $S\subseteq \omega$ with its charateristic function $\chi_S$, 
it follows that 
\begin{equation} \label{La}
\Lambda_{\bm{x}}(\mathcal{I})=\pi_X\left[\left(\psi^{-1}[\,\mathcal{I}^+]\times X\right)\cap \bigcap\nolimits_{m \in \omega}\zeta^{-1}_m[\mathrm{Fin}]\right]. 
\end{equation}

\begin{claim}\label{claim1}
$\psi^{-1}[\,\mathcal{I}^+]\times X$ is a s-analytic subset of $\mathcal{N}\times X$. 
\end{claim}
\begin{proof}
Thanks to Lemma \ref{lem:fafa} and the fact that $\psi$ is continuous, $\psi^{-1}[\,\mathcal{I}^+]$ is an (ordinary) analytic subset of $\mathcal{N}$. Hence there exists an uncountable Polish space $Y$ and a Borel set $B\subseteq \mathcal{N}\times Y$ such that $\psi^{-1}[\,\mathcal{I}^+]=\pi_{\mathcal{N}}[B]$. The claim follows from the fact that 
$$
B\times X\in \mathscr{B}(\mathcal{N}\times Y) \times \mathscr{B}(X) \subseteq \mathscr{B}(\mathcal{N}\times Y\times X),
$$
see e.g. \cite[Proposition 251E]{MR2462280}, and that $\pi_{\mathcal{N}\times X}[B\times X]=\pi_{\mathcal{N}}[B]\times X$.
\end{proof}

\begin{claim}\label{claim2}
For each $\eta \in X$ and $m \in \omega$, the section $\zeta_m(\,\cdot, \eta)$ is continuous. 
\end{claim}
\begin{proof}
Fix $\eta \in X$ and $m \in \omega$. 
It is enough to show that the the section $\zeta_m(\,\cdot, \eta)$ is sequentially continuous. 
For, pick a sequence $(n_k^{(p)}: p \in \omega)$ of elements of $\mathcal{N}$ which is convergent to some $(n_k) \in \mathcal{N}$. 
Then, for each $t \in \omega$, there exists $p_t \in \omega$ such that $n_t^{(p)}=n_t$ for all $p\ge p_t$. 
Hence $\zeta_m((n_k^{(p)}), \eta)(t)= \zeta_m((n_k), \eta)(t)$ for all $p\ge p_t$. This proves that $\lim_p \zeta_m((n_k^{(p)}), \eta)= \zeta_m((n_k), \eta)$. 
\end{proof}

\begin{claim}\label{claim3}
For each $(n_k)\in \mathcal{N}$ and $m \in \omega$, the section $\zeta_m((n_k), \cdot\,)$ is Borel measurable. 
\end{claim}
\begin{proof}
Fix $(n_k)\in \mathcal{N}$ and $m \in \omega$ and define for convenience $\xi: X\to \{0,1\}^\omega$ by $\xi(\eta):=\zeta_m((n_k), \eta)$ for all $\eta \in X$. 
It is enough to show that preimage $\xi^{-1}[U]$ is Borel in $X$ for every basic open set $U\subseteq \mathcal{N}$. 
This is clear if $U=\emptyset$. 
Otherwise there exist $p\in \omega$ and $r_0,r_1,\ldots,r_p \in \{0,1\}$ such that 
$
U=\{a \in \{0,1\}^\omega: a(t)=r_t \text{ for all }t=0,1,\ldots,p\}$. 
Set $A:=\{t \in \{0,1,\ldots,p\}: r_t=0\}$ and $B:=\{t \in \{0,1,\ldots,p\}: r_t=1\}$. It follows that 
\begin{displaymath}
\begin{split}
\xi^{-1}&[U]=\{\eta \in X: \zeta_m((n_k), \eta)(t)=r_t \text{ for all }t=0,1,\ldots,p\}\\
&=\left(\bigcap_{t \in A}\{\eta \in X: \zeta_m((n_k), \eta)(t)=0\right) \cap \left(\bigcap_{t \in B}\{\eta \in X: \zeta_m((n_k), \eta)(t)=1\right)\\
&=\left(\bigcap_{t \in A}\{\eta \in X: x_{n_t} \in U_{\eta,m}\right) \cap \left(\bigcap_{t \in B}\{\eta \in X: x_{n_t} \notin U_{\eta,m}\right)\\
&=\left(\bigcap_{t \in A}\,\bigcup_{\eta \in X:\, x_{n_t} \in U_{\eta,m}} U_{\eta,m}\right) \cap \left(\bigcap_{t \in B}\,\bigcap_{\eta \in X:\, x_{n_t} \notin U_{\eta,m}} U_{\eta,m}^c\right).
\end{split}
\end{displaymath}
Since each $U_{\eta,m}$ is open, we conclude that $\xi^{-1}[U]$ is a Borel set. 
\end{proof}

\begin{claim}\label{claim4}
For each $m \in \omega$, the set $\zeta_m^{-1}[\mathrm{Fin}]$ is Borel. 
\end{claim}
\begin{proof}
Thanks to Claim \ref{claim2}, Claim \ref{claim3}, and \cite[Theorem 3.1.30]{S}, each map $\zeta_m$ is Borel measurable. The conclusion follows since $\mathrm{Fin}$ is a $F_\sigma$-set. 
\end{proof}

To conclude the proof of Proposition \ref{prop:Polish}, we have by Claim \ref{claim1} and Claim \ref{claim4} that both $\psi^{-1}[\,\mathcal{I}^+]\times X$ and $\sigma_m^{-1}[\mathrm{Fin}]$ are s-analytic subsets of $\mathcal{N}\times X$, hence also their intersection by Lemma \ref{lem:countableintersection}. 
By the identity \eqref{La}, $\Lambda_{\bm{x}}(\mathcal{I})$ is the projection on $X$ of a s-analytic subset of $\mathcal{N}\times X$, which is s-analytic by Lemma \ref{lem:projanalytic}. 
\end{proof}

\section{Minimal ideals $\mathcal{I}_W$ and their complexities}\label{sec:minimal}

In this last Section, following the line of research initiated in \cite{Xi, Xi2}, we recall the definition of certain ideals and we study their structural and topological complexity. For, given a sequence $\bm{x}$ taking values in a first countable Hausdorff space $X$ and a subset $W\subseteq X$, define the ideal 
$$
\mathcal{I}_W:=\{A\subseteq \omega: \mathrm{L}_{\bm{x} \upharpoonright A} \cap W= \emptyset\},
$$
where $\mathrm{L}_{\bm{x} \upharpoonright A}:=\Lambda_{\bm{x} \upharpoonright A}(\mathrm{Fin})$, 
see \cite[Section 2]{Xi}. More explicitly, $A \subseteq \omega$ belongs to $\mathcal{I}_W$ if and only if $A$ is finite or, in the opposite, $A$ is infinite and there are no infinite subsets $B\subseteq A$ such that $(x_n: n \in B)$ is convergent to some element of $W$. Note that $\mathcal{I}_W$ may be not admissible: for instance, if $W=\emptyset$ then $\mathcal{I}_W=\mathcal{P}(\omega)$.

The main reason of its introduction is the following: 
\begin{thm}\label{thm:xiIW}
Let $\bm{x}$ be a sequence taking values in a first countable Hausdorff space $X$ and fix a subset $W\subseteq X$. Then 
\begin{equation}\label{eq:Xiclaim}
\Lambda_{\bm{x}}(\mathcal{I}_W)=W \cap \mathrm{L}_{\bm{x}}. 
\end{equation}
In particular, 
if $X$ is separable, then 
$W \in \mathscr{L}(\mathcal{I}_W)$. 
\end{thm}
\begin{proof}
    The first part follows by \cite[Theorem 2.2(i)]{Xi}. This implies, if $W\subseteq \mathrm{L}_{\bm{x}}$, then $\Lambda_{\bm{x}}(\mathcal{I}_W)=W$, so that $W \in \mathscr{L}(\mathcal{I}_W)$. Hence, the second part is obtained by choosing a sequence $\bm{x}$ with dense image. 
\end{proof}

Note that that Theorem \ref{thm:xiIW} proves that, even if $X=\mathbb{R}$, the family $\mathscr{L}(\mathcal{I})$ may contain sets which are not Borel, which answered an open question in \cite{BL}. 

It is also worth to remark that ideals $\mathcal{I}_W$ are not the only ones for which \eqref{eq:Xiclaim} holds, see \cite[Theorem 2.7]{Xi}; cf. also \cite[Theorems 3.4 and 3.5]{Xi2} for further refinements of Theorem \ref{thm:xiIW} in regular and Polish spaces. 
However, we show that they are the smallest ideals with such property. For, we say that $\mathcal{I}$ \emph{contains an isomorphic copy of} $\mathcal{J}$ if there exists a bijection $\phi$ on $\omega$ such that $\phi[\,\mathcal{J}]\subseteq \mathcal{I}$, where $\phi[\,\mathcal{J}]$ stands for the family $\{\phi[S]: S \in \mathcal{J}\}$.

\begin{prop}\label{prop:minimality}
 Let $\bm{x}$ be a sequence taking values in a 
 compact metric space $X$ 
 and fix a dense subset $W\subseteq X$. 
 Let also $\mathcal{I}$ be an ideal on $\omega$ such that $W \in \mathscr{L}(\mathcal{I})$. Suppose also that $\bm{x}$ has dense image, and that
 \begin{equation}\label{eq:conditionconvleKX}
 \forall A \in \mathcal{I}^+, \forall \bm{y} \in X^\omega, \quad 
 \Lambda_{\bm{y}\upharpoonright A}(\,\mathcal{I}\upharpoonright A)\neq \emptyset. 
 \end{equation}
 Then $\mathcal{I}$ contains an isomorphic copy of $\mathcal{I}_W$. 
\end{prop}
\begin{proof}
By hypothesis, there exists a sequence $\bm{y}$ taking values in $X$ such that $\Lambda_{\bm{y}}(\mathcal{I})=W$.  
Let $d$ be a compatible metric on $X$ and note that the denseness of $W$ implies that $\mathrm{L}_{\bm{y}}=X$. 
Since $X$ is compact, there exists a strictly increasing sequence of integers $(\iota_n: n \in \omega)$ such that $\iota_0:=0$ and, for each $n \in \omega$, the family of open balls $\{B(x_k,2^{-n}): k \in [\iota_n,\iota_{n+1})\}$ is an open cover of $X$. 

At this point, define recursively the map $\phi: \omega \to \omega$ as it follows:
\begin{enumerate}[label={\rm (\roman{*})}]
\item $\phi(0)$ is the smallest integer $k\in \omega$ such that $d(y_{\phi(k)}, x_0)<1$; 
\item for each integer $m \in (1,\iota_1)$, $\phi(m)$ is the smallest integer $k\in \omega$ such that $d(y_{\phi(k)}, x_m)<1$ and $k\notin \phi[\{0,1,\ldots,m-1\}]$; 
\item for each $n,m \in \omega$ with $n\ge 1$ and $m \in [\iota_n,\iota_{n+1})$, $\phi(m)$ is the smallest integer $k\in \omega$ such that $d(y_{\phi(k)}, x_m)<2^{-n}$ and $k\notin \phi[\{0,1,\ldots,m-1\}]$. 
\end{enumerate}
It follows by construction that $\phi$ is a bijection on $\omega$ which satisfies 
\begin{equation}\label{eq:xyc0}
\lim_{n\to \infty} d(y_{\phi(n)}, x_n)=0.
\end{equation}

Fix $A \in \mathcal{I}^+$ and note that the set of $\mathcal{I}\upharpoonright A$-limit points of $\bm{y}\upharpoonright A$ is contained both in $\mathrm{L}_{\bm{y}\upharpoonright A}$ and in $\Lambda_{\bm{y}}(\mathcal{I})$. Since the latter is equal to $W$, we obtain by \eqref{eq:conditionconvleKX} and \eqref{eq:xyc0} that 
$$
\emptyset \neq \Lambda_{\bm{y}\upharpoonright A}(\,\mathcal{I}\upharpoonright A)\subseteq \mathrm{L}_{\bm{y}\upharpoonright A} \cap W= \mathrm{L}_{\bm{x}\upharpoonright \phi^{-1}[A]} \cap W.
$$
By the definition of $\mathcal{I}_W$, we get $\phi^{-1}[A] \notin \mathcal{I}_W$. Therefore $\phi[\,\mathcal{I}_W] \subseteq \mathcal{I}$. 
\end{proof}

As remarked in the Introduction, the technical condition \eqref{eq:conditionconvleKX} has been already studied in the literature. For instance, if $X$ is compact, all $F_\sigma$-ideals $\mathcal{I}$ satisfy \eqref{eq:conditionconvleKX}, taking into account \cite[Theorem 2.3]{BL} and \cite[Lemma 3.1(vi)]{MR3920799}.

The following result, due to He et al. \cite{Xi}, deals with the topological complexity of ideals $\mathcal{I}_W$. \emph{Hereafter, we assume for simplicity that $X$ is the Cantor space and $\bm{x}$ is an enumeration of the rationals \textup{(}i.e., finitely supported sequences\textup{)} of $X$, with $x_0:=(0,0,\ldots)$.} 
\begin{thm}\label{thm:XiIWtopology}
Let $W$ be a subset of the Cantor space $X=\{0,1\}^{\omega}$. Then\textup{:}
\begin{enumerate}[label={\rm (\roman{*})}]
\item \label{item:1Xitopology} if $W$ is closed then $\mathcal{I}_W$ is a $\Sigma^0_2$-ideal\textup{;}
\item \label{item:2Xitopology} if $W$ is $\Sigma^0_2$ then $\mathcal{I}_W$ is a $\Pi^0_3$-ideal\textup{;}
\item \label{item:3Xitopology} if $W$ is open then $\mathcal{I}_W$ is an analytic $P$-ideal\textup{.}
\end{enumerate}
\end{thm}
\begin{proof}
    It follows by \cite[Theorem 2.2(ii)]{Xi} 
    by choosing a sequence $\bm{x}$ with dense image (the case $W=\emptyset$ holds as well). 
\end{proof}

Taking into account also the results obtained in Section \ref{sec:large}, we prove some characterizations for the [non]closedness of the subset $W$: 
\begin{thm}\label{thm:onlytwo}
Let $W$ be a nonempty subset of the Cantor space $X=\{0,1\}^{\omega}$ such that $\mathcal{I}_W$ has the hereditary Baire property. Then the following are equivalent\textup{:}
\begin{enumerate}[label={\rm (\roman{*})}]
\item \label{item:1Wclosed} $W$ is not closed\textup{;}
\item \label{item:2Wclosed} $\emptyset \times \mathrm{Fin} \le_{\mathrm{RB}} \mathcal{I}_W$\textup{;}
\item \label{item:3Wclosed} $\mathcal{I}_W$ is $\Pi^0_3$-hard\textup{;}
\item \label{item:4Wclosed} $\Sigma^0_2\subseteq \mathscr{L}(\mathcal{I}_W)$. 
\end{enumerate}
\end{thm}
\begin{proof}
    \ref{item:1Wclosed} $\implies$ \ref{item:2Wclosed}. 
    Thanks to Theorem \ref{thm:xiIW}, $W \in \mathscr{L}(\mathcal{I}_W)$, hence the inclusion $\mathscr{L}(\mathcal{I}_W)\subseteq \Pi^0_1$ fails. The claim follows by Lemma \ref{lem:step1Gdeltasigma}. 
    
    \ref{item:2Wclosed} $\implies$ \ref{item:3Wclosed}. 
    It follows by 
    Lemma \ref{lem:step2Gdeltasigma}. 

    \ref{item:2Wclosed} $\implies$ \ref{item:4Wclosed}. 
    It follows by 
    Corollary \ref{cor:RB}.
    
    \ref{item:3Wclosed} $\implies$ \ref{item:1Wclosed}. 
    If $W$ is closed, then $\mathcal{I}_W$ would be a $\Sigma^0_2$-ideal by Theorem \ref{thm:XiIWtopology}.

    \ref{item:4Wclosed} $\implies$ \ref{item:1Wclosed}. 
    As in the previous implication, if $W$ is closed, then $\mathcal{I}_W$ would be $\Sigma^0_2$. Hence $\mathscr{L}(\mathcal{I}_W)=\Pi^0_1$ by Theorem \ref{thm:oldthm}. However, since $X$ is separable, there exists a countable dense subset, which is $\Sigma^0_2$ and not closed. 
\end{proof}


On the same lines, we are able to characterize the openness of the set $W$ and the $P$-property of the ideal $\mathcal{I}_W$, improving Theorem \ref{thm:XiIWtopology}\ref{item:3Xitopology}:
\begin{thm}\label{thm:onopenIW}
Let $W$ be a subset of the Cantor space $X=\{0,1\}^{\omega}$. Then the following are equivalent\textup{:}
\begin{enumerate}[label={\rm (\roman{*})}]
\item \label{item:1Wopen} $W$ is open\textup{;}
\item \label{item:2Wopen} $\mathcal{I}_W$ is an analytic $P$-ideal\textup{;}
\item \label{item:3Wopen} $\mathcal{I}_W$ is a $P$-ideal\textup{.}
\end{enumerate}
\end{thm}
\begin{proof}
  \ref{item:1Wopen} $\implies$ \ref{item:2Wopen}. It follows by Theorem \ref{thm:XiIWtopology}\ref{item:3Xitopology}.
  
  \ref{item:2Wopen} $\implies$ \ref{item:3Wopen}. This is obvious. 

  \ref{item:3Wopen} $\implies$ \ref{item:1Wopen}. Let us suppose that $W$ is not open, i.e., $W^c$ is not closed. Since $X$ is metrizable, $W^c$ is not sequentially closed. Hence it is possible to pick a sequence $\bm{y}$ taking values in $W^c$ which is convergent to some $\eta \in W$. Let us denote by $d$ a compatible metric on $X$, so that $\lim_k d(y_k,\eta)=0$. 
  Since $\bm{x}$ is dense, for each $k \in\omega$ there exists an infinite set $A_k\subseteq \omega$ such that $\lim \bm{x}\upharpoonright A_k=y_k$. Considering that $y_k \in W^c$, it follows that $\mathrm{L}_{\bm{x}\upharpoonright A_k} \cap W=\emptyset$, hence $A_k \in \mathcal{I}_W$. 

At this point, pick a set $A\subseteq \omega$ such that $A_k\setminus A\in \mathrm{Fin}$ for all $k \in \omega$. Set for convenience $b_{-1}:=0$ and define recursively a sequence of integers $(b_k: k \in \omega)$ such that $b_k$ is the smallest element of $A\cap A_k$ for which $d(y_k,x_{b_k})\le d(y_k,\eta)$ and $b_k>b_{k-1}$ (note that this is well defined). It follows by the triangle inequality that 
$$
\forall k \in \omega, \quad 
d(x_{b_k},\eta) \le d(x_{b_k},y_k)+d(y_k,\eta)\le 2d(y_k,\eta).
$$
Therefore $B:=\{b_k: k \in \omega\}$ is an infinite subset of $A$ such that $\lim \bm{x}\upharpoonright B=\eta$. Since $\eta \in W$, this proves that $\mathrm{L}_{\bm{x}\upharpoonright A} \cap W\neq \emptyset$. Hence, by the definition of $\mathcal{I}_W$, we conclude that $A \notin \mathcal{I}_W$. To sum up, $(A_k: k \in \omega)$ is a witnessing sequence of sets in $\mathcal{I}_W$ which fails the $P$-property for the ideal $\mathcal{I}_W$. 
\end{proof}

As a consequence, we obtain that there are only three possibilities for the complexity of ideals $\mathcal{I}_W$: 
\begin{cor}\label{cor:onlythree}
    Let $W$ be a subset of the Cantor space $X=\{0,1\}^\omega$. Then exactly one of the following cases occurs: 
    \begin{enumerate}[label={\rm (\roman{*})}]
    \item \label{item:1topologyIW} $\mathcal{I}_W$ is a $\Sigma^0_2$-ideal\textup{;}
    \item \label{item:2topologyIW} $\mathcal{I}_W$ is a Borel ideal which is not $\Sigma^0_3$\textup{;}
    \item \label{item:3topologyIW} $\mathcal{I}_W$ is not a Borel ideal. 
    \end{enumerate}
\end{cor}
\begin{proof}
First, if $W=\emptyset$, then $\mathcal{I}_W=\mathcal{P}(\omega)$, which satisfies only \ref{item:1topologyIW}. Hence suppose hereafter that $W$ is nonempty. Suppose also that $\mathcal{I}_W$ is Borel, so that \ref {item:3topologyIW} fails and $\mathcal{I}_W$ has the hereditary Baire property. If $W$ is closed then $\mathcal{I}_W$ is $\Sigma^0_2$ by Theorem \ref{thm:XiIWtopology}, hence \ref{item:1topologyIW} holds. If $W$ is not closed then $W$ is $\Pi^0_3$-hard by Theorem \ref{thm:onlytwo}, which is equivalent to be not $\Sigma^0_3$ by \cite[Theorem 22.10]{K}, hence \ref{item:2topologyIW} holds.
\end{proof}

In light of Theorem \ref{thm:XiIWtopology}, one may ask whether the third case in Corollary \ref{cor:onlythree} really occurs. We answer in the affirmative with our last main result: 
\begin{thm}\label{thm:nonanalytic}
Let $W$ be a Borel subset of the Cantor space $X=\{0,1\}^\omega$ which is not $\Sigma^0_2$. Then $\mathcal{I}_W$ is 
%
not analytic, hence not Borel. 
\end{thm}


Putting together the above results, we have the following consequence:
\begin{cor}\label{cor:equivalencesBorelanalytic}
Let $W$ be a Borel subset of the Cantor space $X=\{0,1\}^\omega$. Then the following are equivalent\textup{:}
    \begin{enumerate}[label={\rm (\roman{*})}]
    \item \label{item:1lastcorollary} $W$ is $\Sigma^0_2$\textup{;}
    \item \label{item:2lastcorollary} $\mathcal{I}_W$ is a $\Pi^0_3$-ideal\textup{;}
    \item \label{item:3lastcorollary} $\mathcal{I}_W$ is a Borel ideal\textup{;}
    \item \label{item:4lastcorollary} $\mathcal{I}_W$ is an analytic ideal\textup{.}
    \end{enumerate}
\end{cor}
\begin{proof}
    \ref{item:1lastcorollary} $\implies$ \ref{item:2lastcorollary}. It follows by Theorem \ref{thm:XiIWtopology}\ref{item:1Xitopology}. 
    
    \ref{item:2lastcorollary} $\implies$ \ref{item:3lastcorollary} $\implies$ \ref{item:4lastcorollary}. They are obvious. 


    \ref{item:4lastcorollary} $\implies$ \ref{item:1lastcorollary}. It follows by Theorem \ref{thm:nonanalytic}. 
\end{proof}

We recall now some definition about trees, which will be needed in the proof of Theorem \ref{thm:nonanalytic}. 

Denote by $\{0,1\}^{<\omega}$ the set of finite $\{0,1\}$-sequences. We say that $a=(a_0,\ldots,a_n) \in \{0,1\}^{<\omega}$ is an extension of $b=(b_0,\ldots,b_m) \in \{0,1\}^{<\omega}$, shortened as $a\subseteq b$, if $n\le m$ and $a_k=b_k$ for all $k\le n$. Given an infinite sequence $z \in \{0,1\}^{\omega}$, we define $z\upharpoonright n:=(z_0,\ldots,z_n) \in \{0,1\}^{<\omega}$ for each $n \in \omega$. A \emph{tree on} $\{0,1\}$ is a subset $T\subseteq \{0,1\}^{<\omega}$ with the property that, $a \in T$ for all $a,b \in \{0,1\}^{<\omega}$ such that $a\subseteq b \in T$. The \emph{body} of a tree $T$ on $\{0,1\}$, denoted by $[T]$, is the set of all its infinite branches, that is, the set of all sequences $z \in \{0,1\}^{\omega}$ such that $z\upharpoonright n \in T$ for all $n \in \omega$. 

A tree $T$ on $\{0,1\}$ is said to be \emph{pruned} if every $a \in T$ has a proper extension, that is, for all $ a\in T$ there exists $b \in T$ such that $a\neq b$ and $a\subseteq b$, see \cite[Definition 2.1]{K}. Accordingly, we define
$$
\mathrm{PTr}_2:=\{T\subseteq \{0,1\}^{<\omega}: T \text{ is a pruned tree }\}.
$$
Identifying a pruned tree with its characteristic function, the set $\mathrm{PTr}_2$ can be regarded as a closed subset of the Polish space $\{0,1\}^{\{0,1\}^{<\omega}}$, see \cite[Exercise 4.32]{K}. Thanks to Alexandrov's theorem \cite[Theorem 3.11]{K}, $\mathrm{PTr}_2$ is a Polish space. 

With the above premises, let $\mathrm{WF}^\star_2$ be the set of pruned trees on $\{0,1\}$ for which every infinite branch contains finitely many ones, that is, 
$$
\mathrm{WF}^\star_2:=\left\{T \in \mathrm{PTr}_2: \forall z \in [T], \forall^\infty n,\,\,  z_n=0\,\right\},
$$
see \cite[Section 33.A]{K}. It is known that $\mathrm{WF}^\star_2$ is a $\Pi^1_1$-complete subset of $\mathrm{PTr}_2$, hence it is conanalytic but not analytic, see \cite[Exercise 27.3]{K}. 

More generally, given a nonempty subset $W$ of the Cantor space $\{0,1\}^\omega$, let $\mathcal{T}(W)$ be the set of pruned trees $T$ on $\{0,1\}$ which contains an infinite branch in $W$, that is, 
$$
\mathcal{T}(W):=\left\{T \in \mathrm{PTr}_2: [T] \cap W \neq \emptyset\right\}. 
$$
It is immediate to check that $\mathcal{T}(W_\mathrm{irr})$ coincides with $\mathrm{WF}^\star_2$, where $W_\mathrm{irr}$ is the set of irrationals $\{z \in \{0,1\}^\omega: \exists^\infty n, \, z_n=1\}$. 

With the above premises, we are ready for the proof of Theorem \ref{thm:nonanalytic}.
\begin{proof}[Proof of Theorem \ref{thm:nonanalytic}]
    We divide the proof in some intermediate steps. Hereafter, as in the statement of the result, $W$ is a given Borel subset of the Cantor space which is not $\Sigma^0_2$.

    \begin{claim}\label{claim1:nonanalytic}
        Fix nonempty subsets $A,B\subseteq \{0,1\}^\omega$ such that $A\le_{\mathrm{W}}B$. Then 
\begin{equation}\label{eq:claimwasgetrees}
\mathcal{T}(A)\le_{\mathrm{W}}\mathcal{T}(B).
\end{equation}
    \end{claim}
    \begin{proof}
        Since $A$ is Wadge reducible to $B$, there exists a continuous map $\phi: \{0,1\}^\omega \to \{0,1\}^\omega$ such that $\phi^{-1}[B]=A$, i.e., $z \in A$ if and only if $\phi(z) \in B$. Let $\mathcal{K}$ be the family of compact subsets of $\{0,1\}^\omega$, endowed with the Vietoris topology. It follows by \cite[Exercise 4.32]{K} that the map $\psi: \mathrm{PTr}_2 \to \mathcal{K}$ defined by 
$$
\forall T \in \mathrm{PTr}_2, \quad 
\psi(T):=[T].
$$
is a well-defined homeomorphism. In addition, its inverse map is given by $\psi^{-1}(K)=\{z\upharpoonright n: z \in K, n \in \omega\}$ for all $K \in \mathcal{K}$. Thanks to \cite[Exercise 4.29(vi)]{K}, the map $\varphi: \mathcal{K}\to \mathcal{K}$ defined by $\Phi(K):=\{\phi(z): z \in K\}$ is continuous. 

At this point, define the function $f: \mathrm{PTr}_2 \to \mathrm{PTr}_2$ by
$$
f:=\psi^{-1} \circ \Phi \circ \psi, 
$$
so that $f(T)=\{\phi(z)\upharpoonright n: z \in [T], n \in \omega\}$ for each pruned tree $T$.
Now, it is sufficient to show that $f$ is a witnessing map for the claimed Wadge reduction \eqref{eq:claimwasgetrees}. For, note that $f$ is a composition of three continuous maps, hence it is continuous. Moreover, the following chain of equivalences holds for each pruned tree $T$ on $\{0,1\}$: 
\begin{displaymath}
\begin{split}
    T \in \mathcal{T}(A) \quad & \text{ if and only if } \quad [T]\cap A=\emptyset, \\
    & \text{ if and only if } \quad \Phi([T])\cap B=\emptyset, \\
    & \text{ if and only if } \quad [f(T)]\cap B=\emptyset, \\
    & \text{ if and only if } \quad f(T) \in \mathcal{T}(B). \\
\end{split}
\end{displaymath}
Therefore $f^{-1}[\mathcal{T}(B)]=\mathcal{T}(A)$, which completes the proof. 
    \end{proof}

\medskip

\begin{claim}\label{claim2:nonanalytic}
$\mathrm{WF}^\star_2 \le_{\mathrm{W}} \mathcal{T}(W)$. 
\end{claim}
\begin{proof}
    Thanks to \cite[Theorem 22.10]{K}, the Borel set $W$ is $\Pi^0_2$-hard, that is, $A\le_{\mathrm{W}} W$ for all $A \in \Pi^0_2$. In particular, $W_{\mathrm{irr}} \le_{\mathrm{W}} W$. Hence by Claim \ref{claim1:nonanalytic}, we conclude that $\mathrm{WF}^\star_2=\mathcal{T}(W_{\mathrm{irr}}) \le_{\mathrm{W}} \mathcal{T}(W)$. 
\end{proof}

\medskip

\begin{claim}\label{claim4:nonanalytic}
$\mathcal{T}(W)\le_{\mathrm{W}} \mathcal{I}_W$. 
\end{claim}
\begin{proof}
We need to show that there exists a continuous map $h: \mathrm{PTr}_2 \to \mathcal{P}(\omega)$ such that 
$$
T \in \mathcal{T}(W) \,\,\,\text{ if and only if }\,\,\,h(T) \in \mathcal{I}_W
$$
for all pruned trees $T$ on $\{0,1\}$. 

First, define the map $f: \{0,1\}^{<\omega} \to \omega$ as follows: for each $s=(s_0,\ldots,s_k) \in \{0,1\}^{<\omega}$, let $f(s)$ be the unique nonnegative integer such that 
$$
z_{f(s)}=(s_0,\ldots,s_k,1,0,0,\ldots).
$$
Of course, $f$ is injective. Recalling that $x_0=(0,0,\ldots)$ is the unique sequence with empty support, it is easy to see the image of $f$ is the set of positive integers. 

At this point, define the map $h: \mathrm{PTr}_2 \to \mathcal{P}(\omega)$ by 
$ 
h(T):=\{f(s): s \in T\}
$ 
for all $T \in \mathrm{PTr}_2$. Since $T$ is an infinite and $f$ is injective, then $h(T)$ is infinite as well. 
It is also not difficult to show that $h$ is continuous: indeed, a basic clopen set containing $h(T)$ is of the type 
$$
V:=\{S \in \omega: f(A)\subseteq S \text{ and }f(B)\cap S=\emptyset\}
$$
for some finite sets $A,B\subseteq \{0,1\}^{<\omega}$ such that every restriction of each $s \in A$ does not belong to $B$. Then, the set $U:=\{T \in \mathrm{PTr}_2: A \subseteq T \text{ and }B\cap T=\emptyset\}$ is an open set such that $h[U]\subseteq V$.

Next, we claim that \begin{equation}\label{eq:claimT}
[T]=\mathrm{L}_{\bm{x} \upharpoonright h(T)}.
\end{equation}

On the one hand, fix $a \in \mathrm{L}_{\bm{x} \upharpoonright h(T)}$. Then there exists a strictly increasing sequence of positive integers $(p_k)$ such that $\lim_k z_{p_k}=a$. For each $k \in \omega$, set $s^k:=f^{-1}(p_k)$ and define $\ell_k \in \omega$ so that $s^k=(s^k_0,s^k_1,\ldots,s^k_{\ell_k})$. Hence 
$$
z_{p_k}=z_{f(s^k)}=(s^k_0,s^k_1,\ldots,s^k_{\ell_k},1,0,0,\ldots).
$$
Fix $n \in \omega$. Then there exists $k_0 \in \omega$ such that $a \upharpoonright n=z_{p_k}\upharpoonright n$ for all $k\ge k_0$. Since $f$ is injective, there exists $k_1\ge \max\{n,k_0\}$ such that $\ell_{k_1} \ge n$. It follows that 
$$
a \upharpoonright n=z_{f(s^{k_1})}\upharpoonright n=(s^{k_1}_0,s^{k_1}_1,\ldots,s^{k_1}_{n})=s^{k_1}\upharpoonright n.
$$
Since $T$ is a tree and $s^{k_1} \in T$, we obtain $a \upharpoonright n \in T$. By the arbitrariness of $n$, it follows that $a$ is an infinite branch of $T$. Therefore $\mathrm{L}_{\bm{x} \upharpoonright h(T)}\subseteq [T]$. 

On the other hand, fix $a \in [T]$, so that $a \upharpoonright n \in T$ for all $n \in \omega$. It follows that $f(a \upharpoonright n) \in h(T)$ and $z_{f(a \upharpoonright n)}=(a_0,\ldots,a_n,1,0,0,\ldots)$ for all $n \in \omega$. It is clear that $\lim_n z_{f(a \upharpoonright n)}=a$. Therefore the opposite inclusion $[T]\subseteq \mathrm{L}_{\bm{x} \upharpoonright h(T)}$ holds. 

Recall that $T \in \mathcal{T}(W)$ if and only if $[T] \cap W=\emptyset$. Thanks to \eqref{eq:claimT}, this is equivalent to $\mathrm{L}_{\bm{x} \upharpoonright h(T)}\cap W=\emptyset$, that is, $h(T) \in \mathcal{I}_W$. 
\end{proof}

\medskip 

To conclude the proof, thanks to Claim \ref{claim2:nonanalytic}, Claim \ref{claim4:nonanalytic}, and the transitivity of $\le_{\mathrm{W}}$, we get $\mathrm{WF}^\star_2 \le_{\mathrm{W}} \mathcal{I}_W$. Hence there exists a continuous map $g: \mathrm{PTr}_2 \to \mathcal{P}(\omega)$ such that 
$$
g^{-1}[\mathcal{I}_W]=\mathrm{WF}^\star_2.
$$
The conclusion follows by the facts that $\mathrm{WF}^\star_2$ is $\Pi^1_1$-complete (hence, not analytic) and that analytic sets are closed under continuous preimages, see e.g. \cite[Proposition 4.1.2]{S}. 
\end{proof}

It is possible to show that $\mathcal{I}_{W_{\mathrm{irr}}}$ is coanalytic, hence $\Pi^1_1$-complete. We leave as open question for the interested reader to check whether $\mathcal{I}_W$ is always coanalytic whenever $W\subseteq \{0,1\}^\omega$ is Borel and not $\Sigma^0_2$.

\section{Acknowledgments}

The authors are grateful to Rafal Filip\'{o}w and Adam Kwela (University of Gdansk, PL) for several discussions related to the results of the manuscript.

\bibliographystyle{amsplain}
\bibliography{ideale}

\providecommand{\MR}[1]{}
\providecommand{\bysame}{\leavevmode\hbox to3em{\hrulefill}\thinspace}
\providecommand{\MR}{\relax\ifhmode\unskip\space\fi MR }
\providecommand{\MRhref}[2]{%
  \href{http://www.ams.org/mathscinet-getitem?mr=#1}{#2}
}
\providecommand{\href}[2]{#2}
\begin{thebibliography}{10}

\bibitem{MR4566746}
M.~Balcerzak, S.~G{\l}{\k{a}}b, and P.~Leonetti, \emph{Another characterization
  of meager ideals}, Rev. R. Acad. Cienc. Exactas F\'is. Nat. Ser. A Mat.
  RACSAM \textbf{117} (2023), no.~2, Paper No. 90, 10. \MR{4566746}

\bibitem{BL}
M.~Balcerzak and P.~Leonetti, \emph{On the relationship between ideal cluster
  points and ideal limit points}, Topology Appl. \textbf{252} (2019), 178--190.
  \MR{3883171}

\bibitem{MR3034318}
P.~Barbarski, R.~Filip\'{o}w, N.~Mro\.{z}ek, and P.~Szuca, \emph{When does the
  {K}at\u{e}tov order imply that one ideal extends the other?}, Colloq. Math.
  \textbf{130} (2013), no.~1, 91--102. \MR{3034318}

\bibitem{Calb1}
J.~Calbrix, \emph{Classes de {B}aire et espaces d'applications continues}, C.
  R. Acad. Sci. Paris S\'er. I Math. \textbf{301} (1985), no.~16, 759--762.
  \MR{817590}

\bibitem{Calb2}
\bysame, \emph{Filtres bor\'eliens sur l'ensemble des entiers et espaces des
  applications continues}, Rev. Roumaine Math. Pures Appl. \textbf{33} (1988),
  no.~8, 655--661. \MR{962412}

\bibitem{MR0396267}
W.~W. Comfort and S.~Negrepontis, \emph{The theory of ultrafilters}, Die
  Grundlehren der mathematischen Wissenschaften, Band 211, Springer-Verlag, New
  York-Heidelberg, 1974. \MR{0396267}

\bibitem{MR1711328}
I.~Farah, \emph{Analytic quotients: theory of liftings for quotients over
  analytic ideals on the integers}, Mem. Amer. Math. Soc. \textbf{148} (2000),
  no.~702, xvi+177. \MR{1711328}

\bibitem{MR2048515}
\bysame, \emph{Luzin gaps}, Trans. Amer. Math. Soc. \textbf{356} (2004), no.~6,
  2197--2239. \MR{2048515}

\bibitem{MR4448270}
R.~Filip\'ow, K.~Kowitz, and A.~Kwela, \emph{Characterizing existence of
  certain ultrafilters}, Ann. Pure Appl. Logic \textbf{173} (2022), no.~9,
  Paper No. 103157, 31. \MR{4448270}

\bibitem{FKL24}
R.~Filipow, A.~Kwela, and P.~Leonetti, \emph{Borel complexity of sets of ideal
  limit points}, manuscript.

\bibitem{MR2462280}
D.~H. Fremlin, \emph{Measure theory. {V}ol. 2}, Torres Fremlin, Colchester,
  2003, Broad foundations, Corrected second printing of the 2001 original.
  \MR{2462280}

\bibitem{MR1343692}
D.~H. Fremlin, R.~A. Johnson, and E.~Wajch, \emph{Countable network weight and
  multiplication of {B}orel sets}, Proc. Amer. Math. Soc. \textbf{124} (1996),
  no.~9, 2897--2903. \MR{1343692}

\bibitem{MR1181163}
J.~A. Fridy, \emph{Statistical limit points}, Proc. Amer. Math. Soc.
  \textbf{118} (1993), no.~4, 1187--1192. \MR{1181163}

\bibitem{MR4330212}
S.~Garcia-Ferreira and O.~Guzm\'{a}n, \emph{More on {MAD} families and
  {$P$}-points}, Topology Appl. \textbf{305} (2022), Paper No. 107871, 10.
  \MR{4330212}

\bibitem{HJR}
R.~W. Hansell, J.~E. Jayne, and C.~A. Rogers, \emph{{$K$}-analytic sets},
  Mathematika \textbf{30} (1983), no.~2, 189--221 (1984). \MR{737176}

\bibitem{Xi}
X.~He, H.~Zhang, and S.~Zhang, \emph{The {B}orel complexity of ideal limit
  points}, Topology Appl. \textbf{312} (2022), Paper No. 108061, 12.
  \MR{4393937}

\bibitem{Xi2}
\bysame, \emph{More on ideal limit points}, Topology Appl. \textbf{322} (2022),
  Paper No. 108324. \MR{4505549}

\bibitem{MR1476761}
G.~Hjorth and A.~S. Kechris, \emph{New dichotomies for {B}orel equivalence
  relations}, Bull. Symbolic Logic \textbf{3} (1997), no.~3, 329--346.
  \MR{1476761}

\bibitem{MR2777744}
M.~Hru\v{s}\'{a}k, \emph{Combinatorics of filters and ideals}, Set theory and
  its applications, Contemp. Math., vol. 533, Amer. Math. Soc., Providence, RI,
  2011, pp.~29--69. \MR{2777744}

\bibitem{MR2849045}
M.~Hru\v{s}\'{a}k and D.~Meza-Alc\'{a}ntara, \emph{Comparison game on {B}orel
  ideals}, Comment. Math. Univ. Carolin. \textbf{52} (2011), no.~2, 191--204.
  \MR{2849045}

\bibitem{MR3692233}
M.~Hru\v{s}\'{a}k, D.~Meza-Alc\'{a}ntara, E.~Th\"{u}mmel, and C.~Uzc\'{a}tegui,
  \emph{Ramsey type properties of ideals}, Ann. Pure Appl. Logic \textbf{168}
  (2017), no.~11, 2022--2049. \MR{3692233}

\bibitem{MR2861027}
M.~Hru\v{s}\'{a}k and J.~L. Verner, \emph{Adding ultrafilters by definable
  quotients}, Rend. Circ. Mat. Palermo (2) \textbf{60} (2011), no.~3, 445--454.
  \MR{2861027}

\bibitem{MR748847}
W.~Just and A.~Krawczyk, \emph{On certain {B}oolean algebras
  {${\mathscr{P}}(\omega )/I$}}, Trans. Amer. Math. Soc. \textbf{285} (1984),
  no.~1, 411--429. \MR{748847}

\bibitem{MR4358610}
V.~Kadets, D.~Seliutin, and J.~Tryba, \emph{Conglomerated filters and
  statistical measures}, J. Math. Anal. Appl. \textbf{509} (2022), no.~1, Paper
  No. 125955, 17. \MR{4358610}

\bibitem{K}
A.~S. Kechris, \emph{Classical descriptive set theory}, Graduate Texts in
  Mathematics, vol. 156, Springer-Verlag, New York, 1995. \MR{1321597}

\bibitem{MR1617463}
\bysame, \emph{Rigidity properties of {B}orel ideals on the integers}, vol.~85,
  1998, 8th Prague Topological Symposium on General Topology and Its Relations
  to Modern Analysis and Algebra (1996), pp.~195--205. \MR{1617463}

\bibitem{MR3920799}
P.~Leonetti and F.~Maccheroni, \emph{Characterizations of ideal cluster
  points}, Analysis (Berlin) \textbf{39} (2019), no.~1, 19--26. \MR{3920799}

\bibitem{Mezathesis}
D.~Meza-Alc\'{a}ntara, \emph{Ideals and filters on countable sets},  (2009),
  PhD Thesis, UNAM M\'{e}xico.

\bibitem{R}
C.~A. Rogers, J.~E. Jayne, C.~Dellacherie, F.~Topsoe, J.~Hoffmann-Jorgensen,
  D.~A. Martin, A.~S. Kechris, and A.~H. Stone, \emph{Analytic sets}, Academic
  Press, London, 1980.

\bibitem{MR1416872}
S.~Solecki, \emph{Analytic ideals}, Bull. Symbolic Logic \textbf{2} (1996),
  no.~3, 339--348. \MR{1416872}

\bibitem{S}
S.~M. Srivastava, \emph{A course on {B}orel sets}, Graduate Texts in
  Mathematics, vol. 180, Springer-Verlag, New York, 1998. \MR{1619545}

\bibitem{MR3624786}
M.~Staniszewski, \emph{On ideal equal convergence {II}}, J. Math. Anal. Appl.
  \textbf{451} (2017), no.~2, 1179--1197. \MR{3624786}

\bibitem{MR0579439}
M.~Talagrand, \emph{Compacts de fonctions mesurables et filtres non
  mesurables}, Studia Math. \textbf{67} (1980), no.~1, 13--43. \MR{579439}

\bibitem{MR1321463}
F.~van Engelen, \emph{On {B}orel ideals}, Ann. Pure Appl. Logic \textbf{70}
  (1994), no.~2, 177--203. \MR{1321463}

\end{thebibliography}

\end{document}